\documentclass[10pt,twoside, a4paper, english, reqno]{amsart}
\usepackage{graphicx, amsmath, varioref, amscd, amssymb,color, bm, stmaryrd, amsthm}
\usepackage{fancybox}
\usepackage[pdftex, colorlinks=true, backref]{hyperref}
\usepackage{mathrsfs}

\numberwithin{equation}{section}
\allowdisplaybreaks

\newtheorem{theorem}{Theorem}[section]
\newtheorem{lemma}[theorem]{Lemma}
\newtheorem{corollary}[theorem]{Corollary}

\newtheorem{proposition}[theorem]{Proposition}
\theoremstyle{definition}

\theoremstyle{remark}
\newtheorem{remark}[theorem]{Remark}

\newcommand{\n}{\|\!|}

\def\pa{\partial}
\def\na{\nabla}

\newcommand{\Div}{\operatorname{div}}

\newcommand{\Grad}{\nabla}
\newcommand{\vr}{\varrho}

\def\E{{\mathscr E}}

\newcommand{\CI}{\mathbb{I}}

\newcommand{\R}{\mathbb{R}}

\newcommand{\Om}{\ensuremath{\Omega}}

\newcommand{\F}{\mathcal{F}}

\newcommand{\opf}[1]{\mathcal{#1}}
\newcommand{\eps}{\epsilon}
\newcommand{\wE}{\widehat E}

\begin{document}

\title[Hydrodynamic limit of the kinetic Cucker-Smale flocking model]
	  {Hydrodynamic limit of the kinetic \\Cucker-Smale flocking model} 

	\author[Karper]{Trygve K. Karper}\thanks{The work of T.K was supported by the Research Council of Norway through the project 205738}

	\address[Karper]{\newline
	Center for Scientific Computation and Mathematical Modeling, University of Maryland, College Park, MD 20742}
	\email[]{\href{karper@gmail.com}{karper@gmail.com}}
	\urladdr{\href{http://folk.uio.no/~trygvekk}{folk.uio.no/\~{}trygvekk}}

	\author[Mellet]{Antoine Mellet}\thanks{The work of A.M was supported by the National Science Foundation  under the Grant  DMS-0901340}

	\address[Mellet]{\newline
	Department of Mathematics, University of Maryland, College Park, MD 20742}
	\email[]{\href{mellet@math.umd.edu}{mellet@math.umd.edu}}
	\urladdr{\href{http://www.math.umd.edu/~mellet}{math.umd.edu/~{}mellet}}

	\author[Trivisa]{Konstantina Trivisa}\thanks{The work of K.T. was supported by the National Science Foundation   under the Grants  DMS-1109397 and DMS-1211638}
	\address[Trivisa]{\newline
	Department of Mathematics, University of Maryland, College Park, MD 20742}
	\email[]{\href{trivisa@math.umd.edu}{trivisa@math.umd.edu}}
	\urladdr{\href{http://www.math.umd.edu/~trivisa}{math.umd.edu/~{}trivisa}}

	\date{\today}

	\subjclass[2010]{Primary:35Q84; Secondary:35D30}

	\keywords{flocking, kinetic equations, hydrodynamic limit, relative entropy, Cucker-Smale, self-organized dynamics}

%\thanks{The work of T.K was supported by the Research Council of Norway through the project 205738}

\maketitle

\begin{abstract}
 The hydrodynamic limit of the kinetic Cucker-Smale  flocking model is investigated. 
The starting point is the  model considered in  \cite{KMT1}, which in addition to %terms describing 
the free-transport of individuals and the Cucker-Smale alignment operator, %for the treatment of long-range interactions 
includes a strong local alignment term. This term was derived in \cite{KMT2} as the singular limit  
of an alignment operator due to Motsch and Tadmor \cite{MT-2011}. 
The model is enhanced with the addition of noise and a confinement potential.
The objective of this work is the rigorous investigation of the singular limit  
 corresponding to strong noise and strong local alignment.  
The proof relies on  the relative entropy method and entropy inequalities which yield the appropriate convergence results.
The resulting limiting system is an Euler-type flocking system.

\end{abstract}
\setcounter{tocdepth}{1}
\tableofcontents

\section{Introduction}
Mathematical models aimed at capturing parts of the flocking behavior 
exhibited by animals such as birds, fish, or insects, 
are currently receiving widespread attention in the mathematical community. 
Many of these models have sprung out in the wake of the seminal 
paper by Cucker and Smale \cite{CS-2007A}. The typical approach is based 
on particle models where each individual follows a simple set of  
rules. To date, the majority of studies on flocking models have been on the 
behavior of the particle model or the corresponding kinetic equation. 

For practical purposes, if the number of individuals in the flock is very 
high, it might be desirable to identify regimes where the complexity 
of the model may be reduced. This paper is a modest contribution 
to this complex task.
The starting point for our study is the following 
kinetic Cucker-Smale equation on $\R^d\times\R^d\times(0,T)$
\begin{equation}\label{eq:eq1}
	f_t + v\cdot \Grad_x f  + \Div_v\left(fL[f]\right)  -\na_x \Phi\cdot \na_vf =\sigma \Delta_v f +\beta \Div_v(f(v -u)).%\quad\mbox{in } \R^d\times\R^d\times(0,T)
\end{equation}
Here, $f:= f(t,x,v)$ is the scalar density of individuals,  $d\geq 1$ is the spatial dimension, $\beta, \; \sigma \geq 0$
are some constants and $\Phi$ is a given confinement potential.
The alignment operator $L$ is the usual Cucker-Smale (CS)  operator, which has the form
\begin{equation}\label{eq:CS}
L[f] = \int_{\R^d} \int_{R^d} K(x,y)  f(y,w) (w-v)\, dw\, dy,
\end{equation}
with $K$ being a smooth symmetric kernel.
The last term in (\ref{eq:eq1}) describes  strong local alignment interactions, where  $u$ denotes the average local velocity,  defined  by 
$$u(t,x) = \frac{\int_{\R^d}fv~dv}{\int_{\R^d}f~dv}.$$ 
This strong alignment term was introduced in \cite{KMT2} as the following singular limit of the Motsch-Tadmor (MT) alignment operator
\begin{equation}\label{eq:MT} 
\lim_{\phi \rightarrow \delta} \tilde L[f] = \lim_{\phi \rightarrow \delta}\frac{\int_{\R^d} \int_{R^d} \phi(x-y)  f(y,w) (w-v)\, dw\, dy}{\int_{\R^d} \int_{R^d} \phi(x-y)  f(y,w)\, dw\, dy} = (u-v).
\end{equation}
Since the MT term is relatively new in the literature, 
a remark on its purpose seems appropriate.
The MT operator was introduced in \cite{MT-2011} to
 correct a deficiency in the standard Cucker-Smale 
model. Specifically, since the CS operator $L[\cdot]$ is weighted by the density, the effect of the term
is almost zero in sparsely populated regions. The MT operator instead weights by a local average density.
To arrive at \eqref{eq:eq1}, the rationale is to let the MT operator govern alignment 
at small scales and the CS operator the large scales. Our equation \eqref{eq:eq1} is 
then obtained in the local limit \eqref{eq:MT} which seems appropriate at the mesoscopic level.

Since the kinetic equation \eqref{eq:eq1} is posed in $2d+1$ dimensions, 
obtaining a numerical solution of \eqref{eq:eq1} is very costly. In fact, 
the most feasible approach seems to be Monte-Carlo methods
using solutions of the underlying particle model with a large number of particles and realizations. 
Consequently, it is of great interest to determine parameter regimes 
where the model may be reduced in complexity. 
The goal of this paper is to study the singular limit of \eqref{eq:eq1} corresponding to strong noise and strong local alignment, 
that is $\sigma, \beta \to \infty$. More precisely, we are concerned with the limit $\eps\to0$ in the following equation:
\begin{equation}\label{eq:eq2}
	f^\eps_t + v\cdot \Grad_x f^\eps  + \Div_v\left(f^\eps L[f^\eps]\right) -\na_x\Phi\cdot \na_v f^\eps   =\frac{1}{\eps}\Delta_v f^\eps +\frac{1}{\eps} \Div_v(f^\eps(v -u^\eps)).%,\quad\mbox{ in } \R^d\times\R^d\times(0,T).
\end{equation}
%$L$ is given by \eqref{eq:CS} and $ u^\eps$ is defined by
%\begin{equation}\label{eq:uu} 
%\vr^\eps(t,x) = \int_{R^d} f^\eps(t,x,v)\, dv ,\qquad \vr^\eps u^\eps = \int_{\R^d} v\, f^\eps(t,x,v)\, dv
%\end{equation}
This scaling can alternatively be obtained from \eqref{eq:eq1} by the change of variables
\begin{equation}
	x = \epsilon \bar x, \qquad t = \epsilon \bar t, \nonumber
\end{equation}
and assuming that
$K(x,y) = \eps \bar K(x-y)$.

In the remaining parts of this paper, we shall establish with rigorous arguments  that 
\begin{equation*}
	f^\eps \rightarrow \vr(t,x)e^\frac{|v-u(t,x)|^2}{2},
\end{equation*}
where $\vr$ and $u$ are the $\epsilon \rightarrow 0$ limits of
\begin{equation*}
	\vr^\eps = \int_{\R^d}f^\eps ~dv, \qquad \vr^\eps u^\eps = \int_{\R^d}f^\epsilon v~dv.
\end{equation*}
As a consequence, we will conclude that the dynamics of $f$ is 
totally described by the following Euler-Flocking system
\begin{align}\label{eq:cont}
	\vr_t + \Div_x (\vr u) &= 0, \\
		(\vr u)_t + \Div_x(\vr u \otimes u) + \Grad_x \vr 
		& =\!\! \int_{\R^d} K(x,y)\vr(x)\vr(y)[u(y)-u(x)]~dy - \vr \Grad_x \Phi. \label{eq:moment}
\end{align}
The result will be precisely stated in Theorem \ref{thm:main} with the proof 
coming up in Section \ref{sec:proof}. The proof is established  via a relative entropy argument providing in addition  a rate of convergence in $\epsilon$. 
The relative entropy method relies on the ``weak-strong" uniqueness principle established 
by Dafermos for systems of conservation laws admitting convex entropy functional \cite{Dafermos-79} (see also \cite{DiPerna-79}).
It has been successfully used to study hydrodynamic limits of particle systems \cite{GLT-2009, Landin-2004, Mourragui-1996,Yau-1991}.
In our case, the result will be somewhat restricted as we need the existence of smooth solutions
to \eqref{eq:cont} which we only know locally in time.

\begin{remark}
An important issue in the study of Cucker-Smale type equations is whether a given model leads to flocking behavior. By flocking, it is usually meant that the velocity $u(x,t)$ converges, for large $t$, to a constant velocity $\bar u$.
%Several studies on flocking  
%determine if a  model leads to flocking.
%By flocking it is usually meant that the velocity $u \rightarrow \bar u$
%as the time becomes large, where $\bar u$ is a constant velocity. 
Here, we note that flocking 
can only occur in \eqref{eq:cont}-\eqref{eq:moment} if the confinement potential and the pressure 
are in balance. 
Indeed, if $(\vr, \bar u)$ (with $\bar u$ constant) solves \eqref{eq:moment}, then
 $\Grad \vr = -\vr\Grad \Phi$ and hence $\vr = \frac{M}{\int e^{-\Phi}~dx}e^{-\Phi}$.	
We thus see that flocking only 
occurs when the density is very diluted and, in particular, does 
not have compact support. 
\end{remark}

\subsection{Formal derivation of (\ref{eq:cont}) - (\ref{eq:moment})}
For the convenience of the reader, let us 
now give the formal arguments for why \eqref{eq:cont} - \eqref{eq:moment}
can be expected in the limit. 
First, we note that to obtain an interesting limit as $\epsilon \rightarrow 0$ in \eqref{eq:eq2},  
the right-hand side should converge to zero
$$\Delta_v f^\epsilon  + \Div_v(f^\epsilon(v -u^\epsilon)) \rightarrow 0.$$
If this is the case,  the limit $f$ can only have the following form
$$
f^\epsilon \rightarrow f(t,x,v) = \vr(t,x)e^{-\frac{|v-u(t,x)|^2}{2}}.
$$
Hence, it seems plausible that the evolution of $f$ (in the limit) 
can be governed by equations for the macroscopic quantities $\vr$ and $u$
alone. 

To derive equations for $\vr$ and $u$, let us first integrate \eqref{eq:eq2} with respect to $v$
\begin{equation}\label{eq:cont0}
	\vr^\eps_t + \Div (\vr^\eps u^\eps) = 0.
\end{equation}
Hence, by assuming the appropriate converge properties 
and passing to the limit, we obtain the continuity equation \eqref{eq:cont}.

To formally derive \eqref{eq:moment}, let us multiply \eqref{eq:eq2} by $v$ and integrating with respect to $v$ to obtain
\begin{equation}\label{eq:mom}
	\begin{split}
		&(\vr^\eps u^\eps)_t + \Div_x \left(\int_{\R^d} (v\otimes v) f^\eps~dv\right) \\
		&\qquad  - \int_{\R^d} K(x,y)\vr^\epsilon(x)\vr^\epsilon(y)(u^\epsilon (x)-u^\epsilon (y))~dy + \vr^\eps \Grad_x \Phi
	= 0.
	\end{split}
\end{equation}
Passing to the limit in \eqref{eq:cont0} and \eqref{eq:mom} (assuming that $f^\eps \rightarrow \vr e^{-\frac{|u-v|^2}{2}}$, $\vr^\eps \rightarrow \vr$ and $u^\eps \rightarrow u$),
we get:
\begin{equation}\label{eq:jalla}
	\begin{split}
			&(\vr u)_t + \Div_x \left(\vr\int_{\R^d} (v\otimes v) e^{-\frac{|v-u|^2}{2}}~dv\right) \\
			&\qquad  - \int_{\R^d} K(x,y)\vr(x)\vr(y)(u (x)-u(y))~dy + \vr \Grad_x \Phi
		= 0.		
	\end{split}
\end{equation}
By adding and subtracting $u$, we discover that
\begin{equation*}
	\begin{split}
		\int_{\R^d} (v\otimes v) e^{-\frac{|v-u|^2}{2}}~ dv 
		&=\int_{\R^d} (u\otimes u) e^{-\frac{|v-u|^2}{2}} + (v-u)\otimes (v-u)e^{-\frac{|v-u|^2}{2}}~dv \\
		&= u\otimes u  + \mathbb{I}.
	\end{split}
\end{equation*}
Inserting this expression in \eqref{eq:jalla} gives \eqref{eq:moment}.
% \begin{align}\label{eq:cont}
% 	\vr_t + \Div_x (\vr u) &= 0, \\
% 		(\vr u)_t + \Div_x(\vr u \otimes u) + \Grad_x \vr 
% 		& =\!\! \int_\Omega K(x,y)\vr(x)\vr(y)[u(y)-u(x)]~dy - \vr \Grad_x \Phi. \label{eq:moment}
% \end{align}
\qed

\subsection*{Organization of the paper:}
The rest of this paper is organized as follows: 
In Section \ref{S2}, we recall some existence results for the kinetic flocking model \eqref{eq:eq1} (these results were proved in \cite{KMT1}) and for the Euler-flocking model \eqref{eq:cont}-\eqref{eq:moment} (a proof of this result is provided in Appendix \ref{S5}).
In Section \ref{S3} we present our main result, which establishes the convergence of weak solutions  of  the   kinetic equation \eqref{eq:eq2} to the strong solution of the Euler-flocking  system \eqref{eq:cont}-\eqref{eq:moment}. 
The proof of the main theorem is then developed in Section \ref{sec:proof}.

\section{Existence theory}\label{S2}
The purpose of this section is to state some existence results upon which our result relies. 
More precisely, the proof of our main result (convergence of \eqref{eq:eq2} to \eqref{eq:cont}-\eqref{eq:moment}) makes use of relative entropy arguments which require the existence of weak solutions of \eqref{eq:eq2} satisfying an appropriate entropy inequality, and the existence of strong solutions to the Euler-Flocking system \eqref{eq:cont}-\eqref{eq:moment} satisfying an entropy equality.
% (as $\eps \to 0$)
% we will utilize 
%a relative entropy argument.  To facilitate 
%this analysis, we shall need the existence of global weak solutions 
%to the kinetic equation \eqref{eq:eq2} and 
%the existence of smooth solutions to the Euler-Flocking system.
Note that the latter result will be obtained only for short time.

Since entropies play a crucial role throughout the paper, we first need to present the entropy equalities and inequalities satisfied by smooth solutions of \eqref{eq:eq2}  and \eqref{eq:cont}-\eqref{eq:moment}.

\subsection{Entropy inequalities}
Solutions of  \eqref{eq:eq1} satisfy an important entropy equality, which was derived in \cite{KMT1}:
We define the entropy
\begin{equation}
	\F(f) =  \int_{\R^{2d}}f \log f + f\frac{|v|^2}{2} + f\Phi\,dv\, dx \label{entropy}
\end{equation}
and the dissipations
\begin{equation}\label{dissipation}
	\begin{split}
		D_1(f)&= \int_{\R^{2d}} \frac{1}{f}\left|\Grad_v f - f(u-v)\right|^2~\, dv\, dx,\\
		D_2(f)&= \frac{1}{2}\int_{\R^d}\int_{\R^d}\int_{\R^d}\int_{\R^d}K(x,y)f(x,v)f(y,w)\left|v - w\right|^2\, dw\, dy\, dv\, dx.
	\end{split}
\end{equation}
The latter is the dissipation associated with the CS operator $L[\cdot]$.
There holds:
\begin{proposition}\label{prop:main}
Assume that $L$ is the alignment operator  given by \eqref{eq:CS}
with $K$ symmetric and bounded.
If $f$ is a  smooth solution of (\ref{eq:eq2}), then $f$ satisfies
\begin{align}\label{eq:entropy1}
		&\partial_t \mathcal{F}(f) + \frac{1}{\eps}D_1(f) + D_2(f) 
		\nonumber\\
		&  \qquad = d \int_{\R^d}\int_{\R^d}\int_{\R^d}\int_{\R^d}K(x,y)f(x,v)f(y,w) \, dw\, dy\, dv\, dx, 
\end{align}
with $\F(\cdot)$, $D_1(\cdot), D_2(\cdot)$ given by \eqref{entropy} and \eqref{dissipation}. 
Furthermore, if the confinement potential  $\Phi$ is non-negative and satisfies
\begin{equation}\label{eq:conff}
\int_{\R^d} e^{-\Phi(x)}\, dx <+\infty,
\end{equation}
then there exists $C$, depending only on $\|K\|_{\infty}$, $\Phi$ and $\int f_0(x,v)\, dx\, dv$, such that
\begin{equation}\label{eq:entropy1.2}
	\begin{split}
		&\partial_t \mathcal{F}(f) + \frac{1}{2\eps}D_1(f) \\ 
		& \qquad + \frac{1}{2} \int_{\R^d} \int_{\R^d} K(x,y)\vr(x)\vr(y)\left|u(x) - u(y) \right|^2\, dy\,dx  
		 \leq C\eps \mathcal{F}(f(t)).
	\end{split}
\end{equation}
\end{proposition}
The first inequality (\ref{eq:entropy1}) shows that the nonlocal alignment term is responsible for some creation of entropy. 
The second inequality (\ref{eq:entropy1.2}) shows that this term can be controlled by $D_1(f)$ and the entropy itself. This last inequality will play a key role in this paper.

The Euler system of equations \eqref{eq:cont}-\eqref{eq:moment} also satisfies a classical entropy equality.
More precisely, if we define
$$ \E(\vr,u) = \int_{\R^d} \vr \frac{u^2}{2}\, dx + \int_{\R^d} \vr \log\vr\,  + \vr \Phi dx,$$
then any smooth solution of \eqref{eq:cont}-\eqref{eq:moment} satisfies
$$\partial_t \E(\vr,u) +  \frac{1}{2}\int_{\R^d}\int_{\R^d}K(x,y)\vr(x)\vr(y)\left|u(x) - u(y)\right|^2~dy\, dx = 0.$$

Note that the entropy $\E$ and $\F$ are related to each other by the relation
$$ \F\left(\frac{\vr}{(2\pi)^{d/2}}e^{-\frac{|v-u|^2}{2}}\right) = \E(\vr, u).$$
Furthermore, we have the following classical minimization principle (consequence of Jensen inequality):
\begin{equation}\label{eq:min}
\E(\vr,u)\leq \F(f),\qquad \mbox{ if } \vr=\int f\, dv, \; \vr u = \int vf\,dv.
\end{equation}
This relation will be important
in the upcoming analysis
when considering the relative entropy 
of solutions to the kinetic equation \eqref{eq:eq2} and solutions 
to \eqref{eq:cont} - \eqref{eq:moment}.

\subsection{Global weak solutions of the kinetic equation}
The existence of a weak  solution for \eqref{eq:eq1} is far from trivial because of the singularity in the definition of $u$.
We will say that a function $f$ satisfying
$$ f \in C(0,T;L^1(\R^{2d}))\cap L^\infty((0,T)\times\R^{2d}), \quad ( |v|^2+\Phi(x)) f \in L^\infty (0,\infty;L^1(\R^{2d})), $$
is a weak solution of (\ref{eq:eq1}) if the following holds:
\begin{equation}\label{eq:weak}
		\begin{split}
				&\int_{\R^{2d+1}} - f \psi_t - vf\Grad_x \psi +f\Grad_x \Phi \Grad_x \psi - fL[f]\Grad_v \psi ~dvdxdt\\
				&\quad + \int_{\R^{2d+1}}\sigma \Grad_v f\Grad_v \psi -\beta f(u-v)\Grad_v \psi~dvdxdt \\
				&\quad = \int_{\R^{2d}} f^0 \psi(0,\cdot)~dvdx,
		\end{split}
	\end{equation}
for any $\psi \in C_c^\infty([0,T)\times \R^{2d})$, where $u$ is such that $j=\vr u$.

\begin{remark}
Note that the definition of $u$ is ambiguous if $\vr$ vanishes (vacuum). We resolve this by defining $u$ pointwise as follows
\begin{equation}\label{eq:uu1} 
u(x,t)=\left\{ \begin{array}{ll}\displaystyle \frac{j(x,t)}{\vr(x,t)} & \mbox{ if }  \vr(x,t) \neq 0 \\[8pt] 0 &  \mbox{ if }  \vr(x,t) = 0 \end{array}\right.
\end{equation}
This gives a consistent definition of $u$ as can be seen from the bound
$$j \leq  \left(\int |v|^2 f(x,v,t)\,  dv\right)^{1/2} \vr^{1/2},$$
yielding $j=0$ whenever $\vr=0$ and so (\ref{eq:uu1}) implies $j=\vr u$.

\item Note also that $u$ does not belong to any $L^p$ space. However, we have
%$$
%\rho u^2 = \frac{j^2}{\rho}\leq \int |v|^2 f(x,v,t)\,  dv
%$$
%and 
$$ \int_{\R^{2d}} |uf|^2\, dx\, dv \leq \|f\|_{L^\infty(\R^{2d})}  \int_{\R^{2d}} |v|^2 f(x,v,t)\,  dv\, dx,$$
so  the term $u f$ in the weak formulation (\ref{eq:weak})  makes sense as a function in $L^2$.
\end{remark}

The existence result we shall utilize in this paper was obtained as the main result in \cite{KMT1}
and is recalled in the following theorem:
\begin{theorem}\label{thm:kinetic}
Assume that $L$ is the alignment operator (CS) given by \eqref{eq:CS} 
with $K$ symmetric  and bounded.
Assume furthermore that $f_0$ satisfies
$$
f_0\in L^\infty(\R^{2d})\cap L^1(\R^{2d}), \quad \mbox{and} \quad  ( |v|^2+\Phi(x)) f_0 \in  L^1(\R^{2d}).
$$
Then, for all $\eps>0$,  there exist a weak solution $f^\eps$ of (\ref{eq:eq2}) satisfying 
\begin{equation}\label{eq:entropyfinal}
\mathcal F(f^\eps(t)) +\int_0^t  \frac{1}{\eps} D_1(f^\eps)+ D_2(f^\eps) \, ds \leq \mathcal F(f_0)
+Ct,
\end{equation}
where the constant $C$ depends only on $\|K\|_{\infty}$, $\Phi$ and $\int f_0(x,v)\, dx\, dv$.
Furthermore, if $\Phi$ satisfies (\ref{eq:conff}), then $f^\eps$ also satisfies
\begin{eqnarray}
&&\!\!\!\!\!\!\!\!\!\!\!\!\mathcal F(f^\eps(t)) + \frac{1}{2\eps} \int_0^tD_1(f^\eps)ds + \frac{1}{2} \int_0^t\int_{\R^{2d}} K(x,y)\vr^\eps(x)\vr^\eps(y)\left|u^\eps(x) - u^\eps(y) \right|^2 dydx ds \nonumber \\
&& \qquad\qquad  \leq \mathcal F(f_0) + C\eps\int_0^t \mathcal F(f^\eps(s))\, ds\label{eq:entropyfinal2}
\end{eqnarray}
for all $t>0$.
\end{theorem}

\subsection{Existence of solutions to the Euler-Flocking system}
As usual with relative entropy methods, our main result will state that the solutions of \eqref{eq:eq2} converge to a strong solution of the asymptotic system \eqref{eq:cont} - \eqref{eq:moment}, provided  such a solution exists. 
It is thus important to prove that  \eqref{eq:cont} - \eqref{eq:moment} has a strong solution, at least for short time.
This is the object of the next theorem (which we state in the case $d=3$):
%Since we are going to justify the  $\epsilon \rightarrow 0$ limit rigorously using a relative entropy type method, it will 
%be essential that the asymptotic system \eqref{eq:cont} - \eqref{eq:moment} has a solution satisfying 
%an appropriate entropy inequality, at least for short time. 
%We shall make use of the following result:
\begin{theorem}\label{thm:Euler}
Let $(\rho_0,u_0)\in H^s(\R^3)$ with $s>5/2$ and $\rho_0(x)>0$ in $\R^3$ and  assume that $\na_x\Phi \in H^s(\R^3)$.
Then, there exist $T^*>0$ and functions $(\vr, u) \in C([0,T^*];H^s(\R^3))\cap C^1((0,T^*);H^{s-1}(\R^3))$, $\rho(x,t)>0$, 
such that $(\vr, u)$ is the unique strong solution of \eqref{eq:cont} - \eqref{eq:moment} for $t\in(0,T^*)$.
Moreover, $(\vr, u)$ satisfies the equality 
\begin{equation}\label{eq:entropyEuler}
	\partial_t  \E(\vr,u)+ \frac{1}{2}\int_{\R^d} \int_{\R^d}  K(x,y)\vr(x)\vr(y)\left|u(x)-u(y)\right|^2~dydx = 0.
%\begin{equation*}
%	\begin{split}
%		&\sup_{t \in (0,T^*)}\int_{\R^d} \vr \log_+ \vr + \vr \frac{|u|^2}{2} + \vr \Phi~dx
%		+\frac{1}{2}\int_{\R^d} K(x,y)\vr(x)\vr(y)\left|u(x)-u(y)\right|^2~dydx \\
%		&\qquad \leq \int_{\R^d} \vr_0 \log \vr_0 + \vr_0 \frac{|u_0|^2}{2} + \vr_0 \Phi~dx.
%	\end{split}
\end{equation}
\end{theorem}

Since the proof of this theorem is rather long and independent of the rest of the paper, we postpone it to the appendix.

Note in particular that the condition $s>5/2$ implies that the solution satisfies
\begin{equation}\label{eq:bdu} 
u \in L^\infty([0,T^*];W^{1,\infty}(\R^3)).
\end{equation}
Furthermore, dividing the momentum equation by $\rho$, we also get:
\begin{equation}\label{eq:bdlog} 
 \na_x \log \rho \in L^\infty([0,T^*]\times \R^3).
 \end{equation}
These two estimates is all the regularity we will need in our main theorem below.

\section{Main result}\label{S3}
With the existence results of the previous section,
we are   ready to state our main result concerning the convergence of  weak solutions of \eqref{eq:eq2} 
to the strong solution $(\vr, u)$ of the Euler-flocking 
system \eqref{eq:moment}-\eqref{eq:cont} 
as $\eps \rightarrow 0$.

\begin{theorem}\label{thm:main}
	Assume that:
	\begin{enumerate}
		\item $f_0$ is of the form 
		$$
		f_0 = \frac{\vr_0(x)}{(2\pi^{d/2})}e^{-\frac{|u_0(x)-v|^2}{2}},
		$$
		with
		$$
		f_0\in L^\infty(\R^{2d})\cap L^1(\R^{2d}), \quad \mbox{and} \quad  ( |v|^2+\Phi(x)) f_0 \in  L^1(\R^{2d}).
		$$
		\item $f^\eps$ is a weak solution of \eqref{eq:eq2} satisfying the entropy inequality \eqref{eq:entropyfinal2} and 
		with initial condition  $f^\eps(0, \cdot) = f_0(\cdot)$.
		\item
	          $T^*>0$ is the maximal time for which there exists a strong solution $(\vr, u)$ 
		to the Euler system of equations \eqref{eq:cont} - \eqref{eq:moment}, with  $\vr_0=\int_{\R^d}f_0~ dv$ and $\vr_0 u_0 = \int_{\R^d}f_0v~dv$ and satisfying \eqref{eq:bdu} and \eqref{eq:bdlog} (Theorem \ref{thm:Euler} gives in particular $T^*>0$ if $\rho_0$ and $u_0$ are regular enough).
	\end{enumerate}
\vspace{0.2cm}
There exists a constant $C > 0$ depending on 
$$
\opf F(f_0),~ \|K\|_{L^\infty},~ \Phi,~ T^*, \quad \| u\|_{L^\infty(0,T^*;W^{1,\infty}(\R^d))}, \text{ and }~ ||\na \log \rho||_{L^\infty((0,T^*)\times\R^d)},
$$
% and 
%$\|dE(\vr , u)\|_{W^{1,\infty}((0,T)\times {\R^d})}$ 
such that
\begin{align}\label{eq:main}
		&\int_0^{T^*}\int_{\R^d} \frac{\vr^\eps}{2}\left|u^\eps - u\right|^2 + \int_{\vr}^{\vr^\eps}\frac{\vr^\eps -z}{z}~dz~  dxdt \\
		&\quad + \frac{1}{2}\int_0^{T^*}\!\!\!\int_{\R^d} \int_{\R^d} K(x,y)\vr^\eps(x)\vr^\eps(y)\left[(u^\eps(x) - u(x)) - (u^\eps(y) - u(y))\right]^2~dxdydt \nonumber\\
		&\quad \leq C\sqrt{\eps},		\nonumber
\end{align}
where 
$$\vr^\eps= \int_{\R^d} f^\eps\, dv , \qquad \vr^\eps u^\eps= \int_{\R^d} vf^\eps\, dv.$$

\noindent
Moreover, any sequence of functions satisfying (\ref{eq:main}) satisfies:
\begin{align*}
	f^\eps &\overset{\eps \rightarrow 0}{\longrightarrow} \vr e^{-\frac{|v-u|^2}{2}} \text{ a.e and }L^1_{loc}(0,T^*; L^1(\R^d \times \R^d)), \\
	\vr^\eps &\overset{\eps \rightarrow 0}{\longrightarrow} \vr \text{ a.e and }L^1_{loc}(0,T^*; L^1(\R^d)), \\
	\vr^\eps u^\eps &\overset{\eps \rightarrow 0}{\longrightarrow} \vr u \text{ a.e and }L^1_{loc}(0,T^*; L^1(\R^d)), \\
	\vr^\eps |u^\eps|^2 &\overset{\eps \rightarrow 0}{\longrightarrow} \vr u^2\text{ a.e and }L^1_{loc}(0,T^*; L^1(\R^d)).
\end{align*}

\end{theorem}

This theorem will be a direct consequence of Proposition  \ref{prop:relative}  below, the proof of which is the object of Section \ref{sec:proof}.

\section{Proof of Theorem \ref{thm:main}}\label{sec:proof}
%\subsection{Relative entropy}
To reduce the amount of notations needed 
in the proof of Theorem \ref{thm:main} it will be preferable 
to write the Euler-Flocking system in terms of the conservative quantities. 
In our case, the conservative quantities are the density $\vr$ and the momentum $P=\vr u$.
If we denote
$$
U=\begin{pmatrix}
		\vr \\ P
	\end{pmatrix},
$$
we can rewrite the system \eqref{eq:cont}-\eqref{eq:moment} as
\begin{equation}\label{def:vec}
	U_t + \Div_x A(U) = F(U).
\end{equation}
The flux and source term are then given by
\begin{equation*}
	A(U) = \begin{pmatrix}
		P & 0 \\
		\frac{P \otimes P}{\vr} & \vr
	\end{pmatrix},\quad
	F(U) = \begin{pmatrix}
		0 \\
		\vr \tilde P - \tilde \vr P - \vr \Grad_x \Phi
	\end{pmatrix},
\end{equation*}
where we have introduced the notation $\tilde g = \int_{\R^d} K(x,y) g(y)\, dy$.
% \begin{equation*}
% 	F(U) = \begin{pmatrix}
% 		0 \\
% 		\vr \tilde P - \tilde \vr P 
% 	\end{pmatrix}.
% \end{equation*}
%The system \eqref{eq:cont}-\eqref{eq:moment} can then be written
%\begin{equation}\label{def:vec}
%	U_t + \Div_x A(U) = F(U).
%\end{equation}
The entropy $E$ corresponding to \eqref{def:vec}  reads
\begin{equation*}
	E(U) = \frac{P^2}{2\vr} + \vr \log \vr + \vr \Phi,
\end{equation*}
and the \emph{relative entropy} is the quantity
$$
	\opf{E}(V|U) = E(V) - E(U) - dE(U)(V-U),
$$
where $d$ stands for the derivation with respect to the variables $(\vr,P)$.

\medskip

For the system  \eqref{eq:cont}-\eqref{eq:moment}, a simple computation yields
\begin{equation*}
	\begin{split}
		-dE(U)(V - U) &=  
		 -\begin{pmatrix}
			-\frac{P^2}{2\vr^2} + \log \vr + 1 + \Phi \\
			\frac{P}{\vr}
		\end{pmatrix}
		\begin{pmatrix}
			q - \vr \\ Q - P
		\end{pmatrix} \\
		&= \frac{q|u|^2}{2} - \frac{\vr |u|^2}{2} + (\vr - q)(\log \vr + 1 + \Phi) 
		 + \vr u^2 -q u v.
	\end{split}
\end{equation*}
Conseqently,  the relative entropy can alternatively be written
\begin{equation}\label{eq:relativecontrol}
	\begin{split}
		\opf{E}(V| U) 
		&= E(V) - E(U) - dE(U)(V - U) \\
		&= q \frac{|v|^2}{2} - \vr \frac{|u|^2}{2} + q\log q - \vr\log \vr + \Phi (q-\vr)\\
		&\qquad +  \frac{q |u|^2}{2} - \frac{\vr |u|^2}{2} + (\vr - q)(\log \vr + 1 + \Phi) 
		 + \vr u^2 - q u v \\
		&= q \frac{\left|v - u\right|^2}{2} + p(q|\vr),
	\end{split}
\end{equation}
where we have introduced the relative pressure
$$ p(q|\vr) =  q\log q - \vr\log \vr+ (\vr -q)(\log \vr + 1) = \int_{\vr}^{q}\frac{q -z}{z}~dz.$$
Note that the relative pressure controls the $L^2$ norm of the difference
\begin{equation}\label{eq:relativepressure}
 p(q|\vr)  \geq \frac{1}{2}\min\left\{\frac{1}{q(x)}, \frac{1}{\vr(x)}\right\} (q(x)-\vr(x))^2.
\end{equation}

With the newly introduced notation, Theorem \ref{thm:main}
can be recast as a direct consequence of the following proposition:
\begin{proposition}\label{prop:relative}
Under the assumptions of Theorem \ref{thm:main}, 
let 
$$
U=\begin{pmatrix}
		\vr  \\ \vr  u 
	\end{pmatrix}
$$ denote the strong solution to the Euler system of equations \eqref{eq:cont}-\eqref{eq:moment} and let
	\begin{equation*}
	U^\epsilon = \begin{pmatrix}
		\vr^\epsilon \\ \vr^\epsilon u^\eps
	\end{pmatrix},
	\qquad 
	\vr^\epsilon = \int_{\R^d}f^\epsilon~ dv, 
	\qquad 
	\vr^\eps u^\eps = \int_{\R^d}f^\epsilon v~ dv,
%	\qquad 
%	u^\epsilon = \frac{P^\epsilon}{\vr^\epsilon}.
\end{equation*}
be the macroscopic quantities corresponding to the weak solution of the kinetic equation \eqref{eq:eq2}.

The following inequality holds:
\begin{equation}\label{eq:final}
	\begin{split}
		& \int_{\R^d} \opf{E}(U^\epsilon | U)(t)\, dx \\
		&\qquad +\frac{1}{2}\int_0^t\int_{\R^d} \int_{\R^d} K(x,y)\vr^\eps(x)\vr^\eps(y)\left[(u^\eps(x) - u(x)) - (u^\eps(y) - u(y))\right]^2\, dx\, dy\, ds \\
		&\qquad \leq C\int_0^t\int_{\R^d} \opf{E}(U^\eps|U)\, dx\, ds + C\sqrt{\eps}.
	\end{split}
\end{equation}
\end{proposition}
The proof of this proposition
relies on several auxiliary results 
which will be stated and proved 
throughout this section. 
At the end of the section, in Section \ref{subsec:prop}
we close the arguments and conclude the proof. 
However, before we continue, let us first convince the reader  that Proposition \ref{prop:relative}
actually yields Theorem \ref{thm:main}.
\begin{proof}[Proof of Theorem \ref{thm:main}]
	Let us for the moment 
	take Proposition \ref{prop:relative} for granted. Then, the main inequality 
	\eqref{eq:main} follows from \eqref{eq:final} and Gronwall's lemma.

We now need to show that \eqref{eq:main} implies the stated convergence.
First, we note that the entropy estimate \eqref{eq:entropyfinal} implies that $f^\eps$ is bounded in $L\log L$ and thus converges weakly to some $f$.

Next, in view of \eqref{eq:relativecontrol} and \eqref{eq:relativepressure}, the main inequality \eqref{eq:main} yields
$$ \int_0^{T^*}\int_{\R^d} \min\left\{ \frac{1}{\rho^\eps},\frac{1}{\rho}\right\} |\rho^\eps -\rho|^2 \, dx\, dt \longrightarrow 0$$
and 
$$ \int_0^{T^*}\int_{\R^d} \rho^\eps  |u^\eps -u|^2 \, dx\, dt \longrightarrow 0.$$

In particular, we get:
\begin{align*}
\int_{\R^d} |\rho^\eps -\rho| \, dx & = \int_{\R^d} \min\left\{ \frac{1}{\rho^\eps},\frac{1}{\rho}\right\}^{1/2} \max\{\rho^\eps,\rho\}^{1/2} |\rho^\eps -\rho| \, dx \\
& \leq  \left(\int_{\R^d} \min\left\{ \frac{1}{\rho^\eps},\frac{1}{\rho}\right\} |\rho^\eps -\rho|^2 \, dx\right)^{1/2}  \left( \int_{\R^d} \max\{\rho^\eps,\rho\}dx\right)^{1/2}  \\
& \leq  \left(\int_{\R^d} \min\left\{ \frac{1}{\rho^\eps},\frac{1}{\rho}\right\} |\rho^\eps -\rho|^2 \, dx\right)^{1/2}  (2M)^{1/2}.
\end{align*}
Hence, we can conclude that
$$	\vr^\eps \overset{\eps \rightarrow 0}{\longrightarrow} \vr \text{ a.e and }L^1_{loc}(0,T^*; L^1(\R^d)), 
$$

Similary, we see that
\begin{equation*}
	\begin{split}
		\int_{\R^d}|\vr^\eps u^\eps - \vr u|~dx 
		&\leq \int_{\R^d}|\vr^\eps(u^\eps - u)| + |(\vr^\eps - \vr)u|~dx \\
		&\leq M^\frac{1}{2}\left(\int_{\R^d}\vr^\eps |u^\eps - u|^2~dx\right) \\
		&\quad\! + \left(\int_{\R^d} \min\left\{ \frac{1}{\rho^\eps},\frac{1}{\rho}\right\} |\rho^\eps -\rho|^2 \, dx\right)^{1/2}  \left( \int_{\R^d} \max\{\rho^\eps,\rho\}u^2~dx\right)^{1/2}.
	\end{split}
\end{equation*}
Consequently, also
$$	\vr^\eps u^\eps \overset{\eps \rightarrow 0}{\longrightarrow} \vr u \text{ a.e and }L^1_{loc}(0,T^*; L^1(\R^d)).
$$

Moreover, by writing
$$ \rho^\eps {u^\eps}^2 - \rho u^2 = \rho^\eps(u^\eps-u)^2 + 2u(\rho^\eps u^\eps -\rho u)+u^2 (\rho-\rho^\eps)$$ 
we easily deduce
	$$ \vr^\eps |u^\eps|^2 \overset{\eps \rightarrow 0}{\longrightarrow} \vr u^2\text{ a.e and }L^1_{loc}(0,T^*; L^1(\R^d)).
$$	 

At this stage, it only remains to prove that $f$ has the stated maxwellian form.
For this purpose, we first send $\epsilon \rightarrow 0$ in the entropy inequality \eqref{eq:entropyfinal}
and use the convergence of $\vr^\eps u^\eps$ and $\vr^\eps |u^\eps|^2$ to obtain
\begin{equation}\label{eq:final}
	\begin{split}
		&\lim_{\eps \rightarrow 0}\int_{\R^d}f^\eps \frac{v^2}{2}+ f^\eps\log f^\eps~dx + \frac{1}{2}\int_0^t\int_{\R^d}K(x,y)\vr(x)\vr(y)(u(y) - u(x))^2~dxdydt \\
		&\qquad \leq \mathcal{E}(\vr_0, u_0),
	\end{split}
\end{equation}
where we have used that $f_0 = \frac{\vr_0}{(2\pi)^{d/2}} e^\frac{-(v-u_0)^2}{2}$ to conclude the last inequality. 
Next, we subtract the entropy equality \eqref{eq:entropyEuler} from \eqref{eq:final} 
to discover
\begin{equation}\label{eq:final2}
	\begin{split}
		& 0 \leq \lim_{\eps \rightarrow 0}\int_{\R^{2d}}f^\eps \frac{v^2}{2}+ f^\eps\log f^\eps~dxdv - \int_{\R^d}\vr\frac{u^2}{2}+ \vr \log \vr~dx  \leq 0,
	\end{split}
\end{equation}
where the first inequality is \eqref{eq:min}. By convexity of the entropy, we conclude that
\begin{equation*}
	f = \frac{\vr}{(2\pi)^\frac{d}{2}} e^{-\frac{(u-v)^2}{2}},
\end{equation*}
which concludes the proof of Theorem \ref{thm:main}.
\end{proof}

%$$
%E(U) = \widehat E(U) + \vr \Phi, \qquad \wE(U) = \frac{P^2}{2\vr} + \vr \log \vr,
%$$ and 

\subsection{The relative entropy inequality}
The fundamental ingredient in the proof 
of Proposition \ref{prop:relative} 
is a relative entropy inequality 
for the system  \eqref{eq:cont} - \eqref{eq:moment}
which we will derive in this subsection.
However, before we embark on the derivation of this inequality, 
we will need some additional identities and simplifications.

First, we recall that $E$ is an entropy  due to the existence an entropy flux function $Q$ such that
\begin{equation}\label{def:Q0}
	d_j  Q_i(U) = \sum_{k}d_j A_{ki}(U)d_k E(U), \quad i,j=1, \ldots, 2.
\end{equation}
We then have
\begin{equation}
	E(U)_t + \Div_x Q(U) = (-\tilde \vr P + \vr \tilde P)\frac{P}{\vr}.
\end{equation}

Since the confinement potential term  $\Phi \vr$ in the entropy is linear, it does not play any role in the relative entropy. It is thus convenient to 
 introduce the reduced entropy functional
$$ 
\wE(U) = \frac{P^2}{2\vr} + \vr \log \vr.
$$
This allows us to treat the contribution of $\Phi\vr$ as a forcing term (which is part of the $F(U)$) in the proof of Proposition \ref{prop:realtive entropy}.
The reduced entropy flux  $\widehat Q(U)$ is then defined by
\begin{equation}\label{def:Q}
	d_j\widehat Q_i(U) = \sum_{k}d_j A_{ki}(U)d_k \wE(U), \quad i,j=1, \ldots, 2.
\end{equation}

We note that $\widehat  Q$ satisfies
\begin{equation}\label{eq:Qh}
\Div_x \widehat Q(u) = d\widehat E(U) \left(\Div_x A(U)\right),
\end{equation}
while the total entropy flux, defined by   $Q(U) = \widehat{Q}(U) + P \Phi$, satisfies
$$
\Div_x Q(U) = \left(\Div_x A(U)
+ \begin{pmatrix}
0 \\
\vr \Grad_x \Phi	
\end{pmatrix}\right)
dE(U).
$$

We shall also need the relative flux:  
\begin{equation}
	\opf{A}(V|U) = A(V) - A(U) - dA(U)(V-U), \nonumber
\end{equation}
where the last term is to be understood as
\begin{equation}
	\left[dA(U)(V-U)\right]_{i} = dA_i(U)\cdot(V-U), \quad i=1. \ldots, d. \nonumber
\end{equation}

%A direct computation yields:
%Observe that 
%\begin{equation}\label{itslinear}
%	\widehat{H}(V|U) :=\widehat E(V) - \widehat E(U) - d\widehat E(U)(V-U) = H(V|U).
%\end{equation}

The key relative entropy inequality is given by the following proposition:
\begin{proposition}\label{prop:realtive entropy}
Let $U=\begin{pmatrix}
		\vr \\ \vr u
	\end{pmatrix}$ be a strong solution of \eqref{def:vec} satisfying \eqref{eq:entropyEuler} and let $V=\begin{pmatrix}
		q \\ Q = q v
	\end{pmatrix}$ be an arbitrary smooth function. 
%Define the following quantities 
%\begin{equation*}
%	 (q, Q) = V,\quad v= \frac{Q}{q}.
%\end{equation*}
The following inequality holds
\begin{equation}
	\begin{split}
		&\frac{d}{dt}\int_{\R^d} \opf{E}(V|U)~dx \\
		&\qquad +\frac{1}{2}\int_{\R^d} \int_{\R^d} K(x,y)q(x)q(y)\left[(v(x) - u(x)) - (v(y) - u(y))\right]^2~dxdy \\
		&\qquad \leq  \int_{\R^d}\left[ \pa_t \wE(V) + Q \Grad_x \Phi + \frac{1}{2}\int_{\R^d}  K(x,y)q(x)q(y)[v(x) - v(y)]^2~dy\right]dx\\
		& \qquad \qquad -\int_{\R^d} \Grad_x(d\wE(U)):\opf{A}(V|U)~dx\\
		&\qquad \qquad 
		-\int_{\R^d} d\wE(U)\left[V_t + \Div A(V) - F(V)\right]~dx \\
		&\qquad \qquad +\int_{\R^d}\int_{\R^d} K(x,y)q(x)(\vr(y)-q(y))[u(y) - u(x)][v(x) - u(x)] ~dxdy
 	\end{split}\nonumber
\end{equation}
\end{proposition}
\begin{remark}
When $K=0$, such an inequality was established by Dafermos \cite{Dafermos-79} for general system of hyperbolic conservation laws.	
\end{remark}

%\begin{lemma}\label{lem:entinq}
%Let $U$ be a smooth solution to \eqref{def:vec} and let $V$ be an arbitrary smooth function. 
%Define the following quantities 
%\begin{equation*}
%	 (q, Q) = V,\quad v= \frac{Q}{q}.
%\end{equation*}
%The following inequality holds
%\begin{equation}
%	\begin{split}
%		&\frac{d}{dt}\int_\Om E(V|U)~dx \\
%		&\qquad +\frac{1}{2}\int_\Om \int_\Om K(x,y)q(x)q(y)\left[(v(x) - u(x)) - (v(y) - u(y))\right]^2~dxdy \\
%		&\qquad \leq \frac{d}{dt} \int_\Om E(V)~dx + \frac{1}{2}\int_\Om \int_\Om K(x,y)q(x)q(y)[v(x) - v(y)]^2~dxdy \\
%		& \qquad \qquad -\int_\Om \Grad_x(dE(U)):\opf{A}(V|U)~dx\\
%		&\qquad \qquad -\int_\Om dE(U)\left[V_t + \Div A(V) - F(V)\right]~dx \\
%		&\qquad \qquad +\int_\Om\int_\Om K(x,y)q(x)(\vr(y)-q(y))[u(y) - u(x)][v(x) - u(x)] ~dxdy
% 	\end{split}
%\end{equation}
%\end{lemma}

In  order to prove  Proposition \ref{prop:realtive entropy},
we will need the following lemma (see Dafermos \cite{Dafermos-79}). 
\begin{lemma}\label{lem:dafermos}
The following integration by parts formula holds
\begin{equation}\label{magic}
	\begin{split}
		&\int_{\R^d} d^2 \hat E(U)\left(\Div_x A(U)\right)(V-U)~dx \\
		&\qquad \qquad= \int_{\R^d} \left(dA(U)(V-U)\right):\left(\Grad_xd  \hat E(U)\right)~dx,
	\end{split}
\end{equation}
where $:$ is the scalar matrix product.
\end{lemma}
\begin{proof}[Proof of Lemma \ref{lem:dafermos}]
The proof of this equality can be found at several places in the literature. For the sake 
of completeness we recall its derivation here (we follow\cite{BV-2005} [p. 1812]{}).

Differentiating \eqref{def:Q} with respect to $U_l$, we obtain the identity
\begin{equation*}
	\sum_{k=1}^d d_ld_k \hat E(U) d_j A_{ki}(U)  =  d_ld_j \hat Q_i(U) - \sum_{k=1}^dd_ld_j A_{ki}(U)d_k \hat E(U).
\end{equation*}	
Using this identity, we calculate
\begin{equation*}
	\begin{split}
		&d^2\hat E(U)\left(\Div_x A(U)\right)(V-U) \\
		&\quad = \sum_{lij=1}^d \left(\sum_{k=1}^d d_ld_k \hat E(U) d_j A_{ki}\right)\frac{\partial U_j}{\partial x_i}(V_l - U_l) \\
		&\quad = \sum_{lij=1}^d d_l d_j \hat Q_i(U)\frac{\partial U_j}{\partial x_i}(V_l - U_l) 
		- \sum_{lijk=1}^dd_ld_j A_{ki}(U)d_k \hat E(U)\frac{\partial U_j}{\partial x_i}(V_l - U_l) \\
		&\quad = \sum_{l}\Div(d_l \hat Q(U)(V_l-U_l))  - \sum_{kl}\Div_x \left(d_lA_k(U)(V_l-U_l)\right)d_k\hat E(U) \\
		&\quad \qquad + \sum_l (\Grad_x V_l - \Grad_x U_l)\cdot \left[\sum_{k} d_lA_k(U)d_k \hat E(U) - d_l\hat Q(U)\right].
	\end{split}
\end{equation*}	
Now, we observe that \eqref{def:Q} implies that the last term  is zero. Thus, we can conclude
\begin{equation*}
	\begin{split}
		&d^2\hat E(U)\left(\Div_x A(U)\right)(V-U) \\
		&\qquad = \Div_x \left(d\hat Q(U)(V-U)\right) - \Div_x \left(dA(U)(V-U)\right)d\hat E(U)
	\end{split}
\end{equation*}
Integrating this identity over $\R^d$ yields \eqref{magic}.
\end{proof}

\begin{proof}[Proof of Proposition \ref{prop:realtive entropy}]
First of all, we recall that
\begin{align*}
	\opf{E}(V|U) & = E(V) - E(U) - dE(U)(V-U)\\
	 & = \wE(V) - \wE(U) - d\wE(U)(V-U).
\end{align*}
We deduce:
\begin{equation}\label{begin}
	\begin{split}
		\frac{d}{dt} \int_{\R^d} \opf{E}(V|U)~dx 
		& = \int_{\R^d} \partial_t \wE(V) - d\wE(U)U_t 
		- d^2\wE(U)U_t(V-U) \\
		&\qquad - d\wE(U)(V_t - U_t)~dx \\
		&=  \int_{\R^d}\partial_t \wE(V) - d^2\wE(U)U_t(V-U) - d \wE(U)V_t~dx \\
		&:= I_1+I_2 + I_3.
	\end{split}
\end{equation}
Since $d^2 E(U) = d^2\hat E(U)$,  formula \eqref{magic} provides
\begin{equation*}
	\begin{split}
		I_2 &= \int_{\R^d} d^2\wE(U)\left[\Div_x A(U) - F(U)\right](V-U)~dx \\
			& = \int_{\R^d} \Grad_x d\wE(U):dA(U)(V-U)~dx - \int_{\R^d} d^2\wE(U)F(U)(V-U)~dx.
	\end{split}
\end{equation*}
By adding and subtracting, and integrating by parts, we find 
\begin{equation*}
	\begin{split}
		I_3 &= -\int_{\R^d} d\wE(U)\left[V_t + \Div A(V) - F(V)\right]~dx \\
		&\qquad - \int_{\R^d} \left(\Grad_x d\wE(U)\right):A(V) + d\wE(U)F(V)~dx
	\end{split}
\end{equation*}
Consequently, 
\begin{equation}\label{i2pi3}
	\begin{split}
		I_2 + I_3 &= -\int_{\R^d} \Grad_x(d\wE(U)):\opf{A}(V|U) + d\wE(U)\left[V_t + \Div A(V) - F(V)\right]~dx \\
		& \qquad - \int_{\R^d}\Grad_x(d\wE(U)):{A}(U)~dx  \\
	&\qquad  - \int_{\R^d} d^2\wE(U)F(U)(V-U) + d\wE(U)F(V)  ~dx := J_1 + J_2 + J_3.
	\end{split}
\end{equation}
Now, using \eqref{eq:Qh}, we get:
\begin{equation}\label{i2iszero}
	\begin{split}
		J_2 &= - \int_{\R^d}\Grad_x(d\wE(U)):{A}(U)~dx \\
		& = \int_{\R^d} d\wE(U)\Div_x\opf{A}(U)~dx = \int_{\R^d} \Div_x \widehat Q(U)~dx = 0.
	\end{split}
\end{equation}

It only remains to compute the term $J_3$, which is the only non standard term, since it includes all the contributions of the forcing term $F(U)$. We now  insert our specific expression of $\wE$ and $U$.
A simple computation yields:
\begin{equation*}
	d\wE(U) = \begin{pmatrix}
		d_\vr \wE(U) \\
		d_P \wE(U)
	\end{pmatrix}
	= \begin{pmatrix}
		-\frac{P^2}{2\vr^2} + \log \vr + 1 \\
		\frac{P}{\vr}
	\end{pmatrix},
	\qquad d^2\wE(U) = \begin{pmatrix}
	* & -\frac{P}{\vr^2}	 \\
	-\frac{P}{\vr^2} & \frac{1}{\vr}	
	\end{pmatrix}.
\end{equation*}
Let us also introduce two functions $(q, Q)$ such that $V = [q, Q]^T$
and define $v = \frac{Q}{q}$. 
\begin{equation*}
	\begin{split}
		-J_3 &= \int_{\R^d} d^2\wE(U)F(U)(V-U) + d\wE(U)F(V)  ~dx \\
		& = \int_{\R^d} -\frac{P}{\vr^2}  \left( \vr \tilde P - \tilde\vr P -\vr \Grad_x\Phi \right)(q-\vr)
			+ \frac{1}{\vr}\left( \vr \tilde P - \tilde\vr P - \vr \Grad_x \Phi\right)(Q-P) \\
			&\qquad + \frac{P}{\vr}\left( q \tilde Q - \tilde q  Q  - q \Grad_x \Phi  \right)~dx \\
		& = \int_{\R^d} \frac{q}{\vr}  \left( \vr \tilde P - \tilde\vr  P\right) \left(\frac{Q}{q}-\frac{P}{\vr}\right) 
			+ \frac{P}{\vr}\left( q \tilde Q - \tilde q  Q    \right) - Q\Grad_x\Phi~dx \\
		& = \int_{\R^d} q \left(  \widetilde{\vr u} - \tilde\vr  u \right) (v-u) +u \left( q \tilde Q - \tilde q  Q    \right) - Q\Grad_x\Phi~dx 
        \end{split}
\end{equation*}
which yields
\begin{equation}\label{J3}
	\begin{split}		
	-J_3 & = \int_{\R^d}\int_{\R^d} K(x,y)q(x)\vr(y)[u(y) - u(x)][v(x) - u(x)] \\
		&\qquad \qquad + K(x,y)q(y)q(x)u(x)[v(y)- v(x)]~dydx. \\
		&= \int_{\R^d} \int_{\R^d} K(x,y)q(x)q(y)[u(y) - u(x)][v(x) - u(x)]~dxdy\\
		&\quad + \int_{\R^d} \int_{\R^d} K(x,y)q(x)q(y)[v(y) - v(x)][u(x) - v(x)]~dxdy \\
		&\quad + \int_{\R^d}\int_{\R^d} K(x,y)q(x)q(y)[v(y) - v(x)]v(x)~dxdy \\
		&\quad + \int_{\R^d} \int_{\R^d} K(x,y)q(x)(\vr(y)-q(y))[u(y) - u(x)][v(x) - u(x)]~dxdy \\
		&\quad - \int_{\R^d} qv\Grad_x \Phi~dx.
	\end{split}
\end{equation}
Using the symmetry of $K$, we see that the first two terms can be written
\begin{align}
		&\int_{\R^d} \int_{\R^d} K(x,y)q(x)q(y)[u(y) - u(x)][v(x) - u(x)]~dxdy \nonumber\\
		&\quad + \int_{\R^d} \int_{\R^d} K(x,y)q(x)q(y)[v(y) - v(x)][u(x) - v(x)]~dxdy \label{J31}\\
		&= \frac{1}{2}\int_{\R^d} \int_{\R^d} K(x,y)q(x)q(y)[u(y) - u(x)][(v(x) - u(x)) - (v(y) - u(y))]~dxdy \nonumber\\
		&\quad +  \frac{1}{2}\int_{\R^d} \int_{\R^d} K(x,y)q(x)q(y)[v(y) - v(x)][(u(x) - v(x)) - (u(y) - v(y))]~dxdy\nonumber\\
		& = \frac{1}{2}\int_{\R^d} \int_{\R^d} K(x,y)q(x)q(y)\left[(v(x) - u(x)) - (v(y) - u(y))\right]^2~dxdy.\nonumber
\end{align}
From the symmetry of $K$, we also easily deduce
\begin{equation}\label{J32}
	\begin{split}
		&\int_{\R^d} \int_{\R^d} K(x,y)q(x)q(y)[v(y) - v(x)]v(x)~dxdy\\
		&\qquad \qquad =- \frac{1}{2}\int_{\R^d} \int_{\R^d} K(x,y)q(x)q(y)[v(x) - v(y)]^2~dxdy.
	\end{split}
\end{equation}
Hence, by setting \eqref{J32} and \eqref{J31} in \eqref{J3}, we discover 
\begin{equation*}
	\begin{split}
		J_3 &=  \frac{1}{2}\int_{\R^d} \int_{\R^d} K(x,y)q(x)q(y)[v(x) - v(y)]^2~dxdy + \int_{\R^d} qv \Grad_x \Phi~dx\\
		&\quad - \frac{1}{2}\int_{\R^d} \int_{\R^d} K(x,y)q(x)q(y)\left[(v(x) - u(x)) - (v(y) - u(y))\right]^2~dxdy \\
		&\quad +\int_{\R^d} K(x,y)q(x)(\vr(y)-q(y))[u(y) - u(x)][v(x) - u(x)] ~dxdy
	\end{split}
\end{equation*}
We conclude the proof by combining the previous identities \eqref{i2iszero}, \eqref{i2pi3} and \eqref{begin}.

\end{proof}

In our proof Proposition \ref{prop:relative},
we will use the following immediate corollary of Proposition \ref{prop:realtive entropy}:
\begin{corollary}\label{cor:ree}
Let $f^\eps$ be a weak solution of \eqref{eq:eq2} satisfying \eqref{eq:entropyfinal2}
 and let 
$$ U^\eps= (\vr^\eps,\vr^\eps u^\eps), \quad \mbox{ with } \vr^\eps= \int_{\R^d} f^\eps\, dv,\quad \vr^\eps u^\eps=\int_{\R^d} v f^\eps\, dv.$$
Let $U=(\vr,\vr u)$ be the strong solution of \eqref{def:vec} satisfying \eqref{eq:entropyEuler}. 
Then the following inequality holds:
\begin{align}\label{eq:relative entropy eps}
		&\frac{d}{dt}\int_{\R^d} \opf{E}(U^\eps|U)~dx \nonumber\\
		&\qquad +\frac{1}{2}\int_{\R^d} \int_{\R^d} K(x,y)\vr^\eps(x)\vr^\eps(y)\left[(u^\eps(x) - u(x)) - (u^\eps(y) - u(y))\right]^2~dxdy \nonumber\\
		&\qquad \leq  \int_{\R^d} \left[ \pa_t \wE(U^\eps) + P^\eps \Grad_x \Phi 
			+ \frac{1}{2}\int_{\R^d}  K(x,y)\vr^\eps(x)\vr^\eps(y)[u^\eps(x) - u^\eps(y)]^2~dy\right] dx \nonumber\\
		& \qquad \qquad -\int_{\R^d} \Grad_x(d\wE(U)):\opf{A}(U^\eps|U)~dx \nonumber\\
		&\qquad \qquad 
		-\int_{\R^d} d\wE(U)\left[U^\eps_t + \Div A(U^\eps) - F(U^\eps)\right]~dx \\
		&\qquad \qquad +\int_{\R^d} \int_{\R^d} K(x,y)\vr^\eps(x)(\vr(y)-\vr^\eps(y))[u(y) - u(x)][u^\eps(x) - u(x)] ~dxdy.\nonumber
\end{align}
\end{corollary}

In order to deduce Proposition \ref{prop:relative} from this corollary,
it remains to show that
\begin{enumerate}
\item The first term in the right hand side in  \eqref{eq:relative entropy eps} is of order $\eps$ when integrated with respect to $t$  (Lemma \ref{lem:negative}).
\item The second term is controlled by the relative entropy itself (in fact we will show that the relative flux is controlled by the relative entropy, see Lemma \ref{lem:flux}).
\item The third term is of order $\mathcal O(\eps)$ (Lemma \ref{lem:kineticapproximation}).
\item The last term can be controlled by the relative entropy (Lemma \ref{lem:theotherguy}).
\end{enumerate}
The rest of this paper is devoted to the proof of these $4$ points.

\subsection{ (1) The first term}

\begin{lemma}\label{lem:negative}
Let $f^\eps$ be the weak solution  of \eqref{eq:eq2} given by Theorem \ref{thm:kinetic}  and let 
$$ U^\eps= (\vr^\eps,\vr^\eps u^\eps), \quad \mbox{ with } \vr^\eps= \int_{\R^d} f^\eps\, dv,\quad \vr^\eps u^\eps=\int_{\R^d} v f^\eps\, dv.$$
Then
\begin{align*}
& \int_0^t  \int_{\R^d} \left[ \pa_t\wE(U^\eps) + P^\eps \Grad_x \Phi 
			+ \frac{1}{2}\int_{\R^d}  K(x,y)\vr^\eps(x)\vr^\eps(y)[u^\eps(x) - u^\eps(y)]^2~dy\right] dx \, ds\\
			&  \qquad\qquad  \qquad\qquad \leq C\eps \int _0^t \opf{F}(f^\eps(s))\, ds
\end{align*}
for all $t\geq 0$.
\end{lemma}
\begin{proof}
First, we write
\begin{align*}
	& \int_0^t  \int_{\R^d}  \pa_t\wE(U^\eps)\, ds  \\
		&\qquad =\int_{\R^d} \widehat{E}(U^\eps)(t)-\widehat{E}(U_0)~dx \\
		&\qquad = 	\int_{\R^d}  \big[ \widehat{E}(U^\eps)(t)- \widehat{\opf{F}}(f^\eps)(t) \big]
		+ \big[\widehat{\opf{F}}(f^\eps)(t) - \widehat{\opf{F}}(f_0) \big]
		 + \big[\widehat{\opf{F}}(f_0) - \widehat{E}(U_0)\big] .
\end{align*}
The well-preparedness of the initial data gives
\begin{equation*}
	\int_{\R^d} \widehat{\opf{F}}(f_0) - \widehat{E}(U_0)~dx = 0,
\end{equation*}
and \eqref{eq:min} implies
\begin{equation*}
	\int_{\R^d} \widehat{E}(U^\eps)(t)- \widehat{\opf{F}}(f^\eps)(t)~dx \leq 0.
\end{equation*}
Finally, we deduce  (using \eqref{eq:entropyfinal2})
\begin{align*}
	 \int_0^t  \int_{\R^d}  \pa_t\wE(U^\eps)\, ds &  \leq \int_{\R^d} \widehat{\opf{F}}(f^\eps)(t) - \widehat{\opf{F}}(f_0)~dx\\
&  \leq \int_{\R^d} {\opf{F}}(f^\eps)(t) - {\opf{F}}(f_0)~dx- \int_{\R^d} \rho^\eps(t)\Phi - \rho_0 \Phi\, dx \\
	& \leq -\frac{1}{2}\int_0^t\int_{\R^d} \int_{\R^d} K(x,y)\vr^\eps(x)\vr^\eps(y)|u^\eps(y)- u^\eps(x)|^2\, dydxds\\ 
	& \qquad - \int_0^t\int_{\R^d} P^\eps \Grad \Phi\, dx\, ds + C\eps \int_0^t \opf{F}(f^\eps(s))\, ds
\end{align*}	
%Finally, we have
%\begin{equation}
%	\begin{split}
%		&\frac{1}{2}\int_0^t\int_\Om \int_\Om \int_{\R^d}\int_{\R^d}K(x,y)f^\eps(x)f^\eps(y)|v-w|^2~dvdwdxdydt \\
%		&\qquad \qquad \leq \frac{1}{2}\int_0^t\int_\Om \int_\Om K(x,y)\vr^\eps(x)\vr^\eps(y)|u^\eps(y)- u^\eps(x)|^2~ dydxdt
%	\end{split}
%\end{equation}
%the entropy dissipation (Lemma \ref{lem:entropy}) provides
%\begin{equation}\label{eq:prop2}
%	\begin{split}
%		&\int_\Om \widehat{\opf{F}}(f^\eps)(t) - \widehat{\opf{F}}(f^\eps)(0)~ dx + \int_0^t\int_\Om \vr^\eps u^\eps \Grad \Phi~ dxdt\\
%		&\qquad = -\frac{1}{2}\int_0^t\int_\Om \int_\Om K(x,y)\vr^\eps(x)\vr^\eps(y)|u^\eps(y)- u^\eps(x)|^2~ dydxdt.
%	\end{split}
%\end{equation}
\end{proof}

\subsection{(2) Control of the relative flux}
We note that $|\na_x d\hat E(U)|$ is bounded in $L^\infty$  by $|| u||_{L^\infty(0,T^*;W^{1,\infty})}$, $||\na \log \rho||_{L^\infty}$, so the second term in \eqref{eq:relative entropy eps} will be controlled if we prove the following standard lemma:
\begin{lemma}\label{lem:flux}
The following inequality holds for all $U$, $V$:
$$ \int_{\R^d} |\opf{A}(V|U)|\, dx \leq   \int_{\R^d} \opf{E}(V | U)~dx$$
\end{lemma}

\begin{proof}
A straightforward computation gives
\begin{equation*}
	\begin{split}
		&dA(U)(V-U) \\
		&\qquad =
		\begin{pmatrix}
			Q-P  \\
			-\frac{(q-\vr)}{\vr^2}P \otimes P  + \frac{1}{\vr}P\otimes (Q-P) + \frac{1}{\vr}(Q-P)\otimes P +  (q-\vr)		
		\end{pmatrix},
	\end{split}
\end{equation*}
and using the fact that $P=\vr u$ and $Q=qv$, we get:
\begin{equation*}
	\begin{split}
	&	\frac{1}{q}Q\otimes Q - \frac{1}{\vr}P\otimes P -\frac{1}{\vr}P\otimes (Q-P) - \frac{1}{\vr}(Q-P)\otimes P + \frac{(q-\vr)}{\vr^2}P \otimes P\\
	&\qquad  = qv \otimes v + \vr u \otimes u - u \otimes qv - qv \otimes u + (q-\vr)u \otimes u\\
	&\qquad  = q (v-u)\otimes (v-u). 
	\end{split}
\end{equation*}
Since all the other terms in $\opf{A}(\cdot | \cdot)$ are linear, we deduce
\begin{equation*}
	\opf{A}(V|U)
	= \begin{pmatrix}
		0 & 0 \\
		q(v-u)\otimes (v-u) & 0
	\end{pmatrix}.
\end{equation*}
This implies
\begin{equation*}
	\int_{\R^d} \left|\opf{A}(V|U) \right|~dx = \int_{\R^d} q|v-u|^2~dx \leq \int_{\R^d} \opf{E}(V|U)~dx,
\end{equation*}
which concludes our proof.

\end{proof}

\subsection{(3) Kinetic approximation}
%We now consider $U=(\vr,\vr u)$ solution  of the limiting equation, and $U^\eps=(\vr^\eps, \vr^\eps u^\eps)$ defined by
%$$ \vr^\eps = \int_{\R^d} f^\eps(t,x,v)\, dv   , \qquad \vr^\eps u^\eps = \int_{\R^d} v f^\eps(t,x,v)\, dv$$
%\begin{lemma}\label{lem:}
%The following relation holds
%\begin{equation}\label{eq:entineq}
%	E(U^\eps) \leq \mathcal{F}(f^\eps).
%\end{equation}
%\end{lemma}
%\begin{proof}
%Let $g$ be a function of the form
%\begin{equation}
%	g(t,x,v) = h(t,x)e^{-\frac{|v|^2}{2}},
%\end{equation}
%and introduce the probability measure 
%\begin{equation*}
%	d\mu(v) = \frac{1}{(2\pi)^\frac{d}{2}}e^{-\frac{|v|^2}{2}}dv.
%\end{equation*}
%From Jenssen's inequality, we have that
%\begin{equation}
%	\begin{split}
%		\int_\Om\int_{\R^d} g \log g~dxdv
%		& =\int_\Om\int_{\R^d} he^{-\frac{v^2}{2}} \log h - h \frac{v^2}{2}e^{-\frac{v^2}{2}}~dxdv  \\
%		& = \int_\Om \int_{\R^d} h \log h ~d\mu(v)dx - CM
%	\end{split}
%\end{equation}	
%\begin{equation}
%	\begin{split}
%		\vr^\eps |u^\eps|^2 = \int f^\eps vu^\eps~ dv 
%		&\leq \left(\int f^\eps v^2~dv\right)^\frac{1}{2}\left(\int f^\eps |u^\eps|^2~dv\right)^\frac{1}{2} \\
%		&= \left(\int f^\eps v^2~dv\right)^\frac{1}{2}\left(\vr^\eps |u^\eps|^2\right)^\frac{1}{2}
%	\end{split}
%\end{equation}
	
%\end{proof}

\begin{lemma}\label{lem:kineticapproximation}
Let $U$ be a smooth function, let $f^\eps$ be a weak solution of \eqref{eq:eq2} satisfying \eqref{eq:entropyfinal2}
 and define
$$ U^\eps= (\vr^\eps,\vr^\eps u^\eps), \quad \mbox{ with } \vr^\eps= \int_{\R^d} f^\eps\, dv,\quad \vr^\eps u^\eps=\int_{\R^d} v f^\eps\, dv.$$
There exists a constant $C$ depending on $T$,
$|| u||_{L^\infty(0,T^*;W^{1,\infty})}$ and $||\na \log \rho||_{L^\infty}$
such that
$$
\left| \int_0^t \int_{\R^d} dE(U)\left[U^\eps_t + \Div A(U^\eps) - F(U^\eps)\right]~dx \right| \leq
C \left(\int_0^t D_1(f^\eps)\, ds \right)^{1/2}
\leq  C\sqrt{\epsilon}.
$$
\end{lemma}

\begin{proof}
	By setting $\phi := \phi(t,x)$ in \eqref{eq:weak}, we see that
	\begin{equation}
		\begin{split}
			\vr_t^\eps + \Div(\vr^\eps u^\eps) = 0 \quad \text{ in }D'([0,T)\times \Om).
		\end{split}
	\end{equation}
	Setting $\phi := v \Psi(t,x)$, where $\Psi$ is a smooth vector field, we find that
	\begin{equation}
		\begin{split}
			&(\vr^\eps u^\eps)_t + \Div_x (\vr^\eps u^\eps \otimes u^\eps) + \Grad_x \vr^\eps \\
			&\qquad - \int_{\R^d} K(x,y)\vr^\eps(t,x)\vr^\eps(t,y)(u^\eps(x)- u^\eps(y))~dy + \vr^\eps \Grad_x \Phi\\
			&\qquad =  \Div_x\int_{\R^d} \left(u^\eps \otimes u^\eps - v \otimes v + \mathbb{I}\right) f^\eps ~dv,
		\end{split}
	\end{equation}
	in the sense of distributions on $[0,T) \times \Om$. Hence, we have that
	\begin{equation}\label{ax:1}
		\begin{split}
			&\left| \int_0^t \int_{\R^d} dE(U)\left[U^\eps_t + \Div A(U^\eps) - F(U^\eps)\right]~dxdt \right| 	\\
			&\qquad \leq\left|\int_0^t \int_{\R^d} |\na _x dE(U)|
			\left|\int_{\R^d} \left(u^\eps \otimes u^\eps - v \otimes v + \mathbb{I}\right) f^\eps ~dv\right|dxdt\right|	\\
			&\qquad \leq C\int_0^t\int_{\R^d}\left|\int_{\R^d} \left(u^\eps \otimes u^\eps - v \otimes v + \mathbb{I}\right) f^\eps ~dv\right|dxdt,
		\end{split}
	\end{equation}
	where the constant $C$ depends on $|| u||_{L^\infty(0,T^*;W^{1,\infty})}$ and $||\na \log \rho||_{L^\infty}$. To conclude, we have 
	to prove that the righthand side can be controlled by the dissipation. 
	As in \cite{MV-2008}, we calculate
	\begin{equation*}\label{bydiss}
		\begin{split}
			&\int_{\R^d}(u^\eps \otimes u^\eps - v \otimes v + \mathbb{I})f^\eps~dv \\
			& = \int_{\R^d} \left(u^\eps \otimes (u^\eps-v) + (u^\eps-v) \otimes v + \mathbb{I}\right)f^\eps~dv \\
			&= \int_{\R^d} u^\eps \sqrt{f^\eps} \otimes \left((u^\eps-v)\sqrt{f^\eps} - 2\Grad_v \sqrt{f^\eps} \right) 
			+ u^\eps \otimes \Grad_v f^\eps \\
			&\qquad + \left((u^\eps-v)\sqrt{f^\eps} - 2\Grad_v \sqrt{f^\eps} \right)\otimes v\sqrt{f^\eps} +\Grad_v f^\eps\otimes v + \mathbb{I}f^\eps~dv.
		\end{split}
	\end{equation*}
	Using integration by parts, we see that
	\begin{equation*}
		\int_{\R^d}u^\eps \otimes \Grad_v f^\eps~dv = 0, \qquad \int_{\R^d}\Grad_v f^\eps\otimes v~dv = \int_{\R^d} -f \mathbb{I}~dv.
	\end{equation*}
	By applying this and the H\"older inequality to \eqref{bydiss} we find
	\begin{align}\label{ax:2}
			&\int_0^t \int_{\R^d}\left|\int_{\R^d}(u^\eps \otimes u^\eps - v \otimes v + \mathbb{I})f^\eps~dv\right|~dx\, ds \\
			 &\qquad \leq\int_0^t \left(\int_{\R^d} \int_{\R^d}f^\eps|v|^2 + f^\eps|u^\eps|^2~dvdx \right)^\frac{1}{2}D_1(f^\eps)^\frac{1}{2}\, ds
			\leq C\left( \int_0^t D_1(f^\eps)\right)^\frac{1}{2}, \nonumber
	\end{align}
	where the last inequality follows from the entropy bound (Proposition \ref{prop:main}) and 
	\begin{equation*}\label{eq:ent}
		\begin{split}
			\vr^\eps |u^\eps|^2 = \int f^\eps vu^\eps~ dv 
			&\leq \left(\int f^\eps v^2~dv\right)^\frac{1}{2}\left(\int f^\eps |u^\eps|^2~dv\right)^\frac{1}{2} \\
			&= \left(\int f^\eps v^2~dv\right)^\frac{1}{2}\left(\vr^\eps |u^\eps|^2\right)^\frac{1}{2}.
		\end{split}
	\end{equation*}
	We conclude by combining \eqref{ax:1} and \eqref{ax:2}.
\end{proof}

\subsection{(4) The last term}
Finally, we have:
\begin{lemma}\label{lem:theotherguy}
Assuming that $U=(\vr,\vr u)$  and $V=(q,qv)$ are such that $u\in L^\infty(\R^d)$, $q$, $\vr \in L^1(\R^d)$, there exists a constant $C$ such that
\begin{equation}
	\begin{split}
		&\int_{\R^d}\int_{\R^d} K(x,y)q(x)(\vr(y)-q(y))[u(y) - u(x)][v(x) - u(x)] ~dxdy\\
		&\qquad\qquad \leq C \|u\|_{L^\infty} (\| \vr\|_{L^1}+\|q\|_{L^1}) \int_{\R^d} \opf{E}(V | U)~dx
	\end{split}
\end{equation}
\end{lemma}

\begin{proof}
We have 
\begin{equation*}
	\begin{split}
		&\int_{\R^d}\int_{\R^d} K(x,y)q(x)(\vr(y)-	q(y))[u(y) - u(x)][v(x) - u (x)] ~dxdy\\
		&\qquad\qquad \leq 	
		2\|u\|_{L^\infty} \int_{\R^d} \int_{\R^d} K(x,y)q(x)\min\left\{\frac{1}{q(y)}, \frac{1}{\vr(y)}\right\}^{1/2} |\vr(y)-q(y)| \\
		 & \qquad\qquad \qquad\qquad\qquad\qquad\qquad
		 \max\left\{ q(y),\vr(y)\right\}^{1/2} |v(x) - u(x)| ~dxdy \\
		& \qquad\qquad \leq 
		2\|u\|_{L^\infty} 
	 	\left( \int_{\R^d}\int_{\R^d} K(x,y)q(x)\min\left\{\frac{1}{q(y)}, \frac{1}{\vr(y)}\right\} (\vr(y)-q(y))^2 ~dxdy \right)^{1/2} \\
		 & \qquad\qquad \qquad\qquad \left( \int_{\R^d} \int_{\R^d} K(x,y)q(x) \max\left\{q(y),\vr(y)\right\} |v(x) - u(x)| ^2 ~dxdy \right)^{1/2}		 	
		\end{split}
\end{equation*}
And so using (\ref{eq:relativepressure}) and the fact that $K(x,y)\leq C$, we deduce:
\begin{equation*}
	\begin{split}
		&\int_{\R^d} \int_{\R^d} K(x,y)q(x)(\vr(y)-q(y))[u(y) - u(x)][v(x) - u (x)] ~dxdy\\
		& \qquad\qquad \leq 
		C\|u\|_{L^\infty} \|q\|_{L^1}^{1/2} (\|q\|_{L^1}+\|\vr\|_{L^1})^{1/2}
	 	\left( \int_{\R^d} p(q|\vr)(y) ~dy \right)^{1/2} \\
		 & \qquad\qquad \qquad\qquad \qquad\qquad\times \left( \int_{\R^d} q(x) [v(x) - u(x)] ^2 ~dx \right)^{1/2}\\
    		& \qquad\qquad \leq 
		C\|u\|_{L^\infty} \|q\|_{L^1}^{1/2} (\|q\|_{L^1}+\|\vr\|_{L^1})^{1/2}\int_{\R^d} \opf{E}(V | U)~dx.
\end{split}
\end{equation*}
	 	
\end{proof}

\subsection{Proof of Proposition \ref{prop:relative}}\label{subsec:prop}
We recall that
$$ \opf{E}(U^\eps(0)|U(0))=0.$$
For any $t \in (0,t)$, we integrate \eqref{eq:relative entropy eps} over $(0,t)$ and apply Lemmas \ref{lem:flux}, \ref{lem:kineticapproximation} and \ref{lem:theotherguy}, and , 
 to obtain the inequality
\begin{align}\label{eq:prop1}
		&\int_{\R^d} \opf{E}(U^\eps|U)(t)~dx \nonumber\\
		&\qquad +\frac{1}{2}\int_0^{t}\int_{\R^d} \int_{\R^d} K(x,y)\vr^\eps(x)\vr^\eps(y)\left[(u^\eps(x) - u(x)) - (u^\eps(y) - u(y))\right]^2\, dx\, dy\, ds \nonumber\\
		&\qquad \qquad \leq  \int_0^t\int_{\R^d} \pa_t  \widehat{E}(U^\eps) + P^\eps \Grad_x \Phi 
			+ \frac{1}{2}\int_{\R^d}  K(x,y)\vr^\eps(x)\vr^\eps(y)[u^\eps(x) - u^\eps(y)]^2\, dy\, dx\, ds  \nonumber\\
		&\qquad \qquad\quad  + C\sqrt{\eps} +C
		\int_0^t \int_{\R^d} \opf{E}(U^\eps | U)\, dx\, dt. \nonumber
\end{align}
Lemma \ref{lem:negative} now implies:
\begin{align}
		&\int_{\R^d} \opf{E}(U^\eps|U)(t)~dx \nonumber\\
		&\qquad +\frac{1}{2}\int_0^{t}\int_{\R^d} \int_{\R^d} K(x,y)\vr^\eps(x)\vr^\eps(y)\left[(u^\eps(x) - u(x)) - (u^\eps(y) - u(y))\right]^2~dxdydt \nonumber\\
		&\qquad \leq C\eps \int_0^t \opf{F}(f^\eps(s))\, ds +   C\sqrt{\eps}+ C\int_0^t\int_{\R^d} \opf{E}(U^\eps|U)~dxdt.
\end{align}
and \eqref{eq:entropyfinal2} implies
$$ \int_0^t \opf{F}(f^\eps(s))\, ds  \leq  e^{C\eps t}.$$
We deduce
\begin{align}
		&\int_{\R^d} \opf{E}(U^\eps|U)(t)~dx \nonumber\\
		&\qquad +\frac{1}{2}\int_0^{t}\int_{\R^d} \int_{\R^d} K(x,y)\vr^\eps(x)\vr^\eps(y)\left[(u^\eps(x) - u(x)) - (u^\eps(y) - u(y))\right]^2~dxdydt \nonumber\\
		&\qquad \leq   C(T)\sqrt{\eps}+ C\int_0^t\int_{\R^d} \opf{E}(U^\eps|U)~dxdt.
\end{align}

which completes the proof of Proposition \ref{prop:relative}.

\newpage 

\appendix

\section{Local well-posedness of the Euler-flocking system}\label{S5}
The purpose of this appendix is to prove 
the existence of a local-in-time unique smooth solution of
the Euler-flocking equations. In particular, 
the objective is to prove Theorem \ref{thm:Euler}
which we relied upon to prove our main result (Theorem \ref{thm:main}).

First, we observe that the system \eqref{eq:cont}-\eqref{eq:moment} is a $4\times 4$ system of conservation laws
which can be written in the following equivalent form when the solution is smooth:
\begin{eqnarray}
\begin{cases}
\partial_t \vr+ u \Grad_x \vr + \vr \Div_x u=0\\
\vr ( \partial_t u + u \Grad_x u) + \Grad_x \vr = F(\vr,u, \Grad_x \Phi).
\end{cases}\label{Euler-smooth}
\end{eqnarray}
Here, 
\begin{equation}\label{potential}
F(\vr, u, \Grad \Phi(x))= \int_{\R^d} K(x,y)\vr(x)\vr(y)[u(y)-u(x)]~dy - \vr \Grad_x \Phi.
\end{equation}

In what follows, we shall need the Sobolev space $H^s(\R^d)$ given by the norm
$$\|g\|_s^2 = \sum_{|\alpha| \le s} \int_{\R^d}|D^{\alpha} g|^2 dx.$$

For $g \in L^{\infty}([0,T]; H^s),$ define
$$\n g \n_{s,T} = \sup_{0 \le t \le T} \|g(\cdot, t)\|_s.$$

Now, consider the Cauchy problem of \eqref{Euler-smooth} with smooth initial data:
\begin{equation}\label{ID}
(\vr, u)|_{t=0} = (\vr_0, u_0)(x).
\end{equation}
The values of the  vector  $$
w = \begin{pmatrix}
\vr\\
u  
 \end{pmatrix}
 $$
 lie in the state space ${\mathcal G},$ which is an open set in $\R^4.$  
The state space ${\mathcal G}$ is introduced because certain physical quantities such as density should be positive. 
Indeed, by invoking the method of characteristics for the continuity equation the following lemma holds:
\begin{lemma}
If $(\vr, u) \in C^1(\R^3 \times [0,T])$ is a uniformly bounded solution of \eqref{eq:cont} with $\vr(x,0)  >0,$ then $\vr(x,t) >0$ on $\R^3 \times [0,t].$
\end{lemma}
System \eqref{Euler-smooth} can be written in the form
$$
\partial_t 
\begin{pmatrix}
\vr\\
u  
 \end{pmatrix}
+ u \nabla_x \begin{pmatrix}
\vr\\
u  
 \end{pmatrix}
+ 
\begin{pmatrix}
\vr \Div_x u \\
\frac{\nabla_x \vr}{\vr}  
 \end{pmatrix}
=
\begin{pmatrix}
0\\
\frac{1}{\vr}  F(\vr, u, \nabla_x \Phi(x)
 \end{pmatrix}
$$
or equivalently
\begin{eqnarray}
\partial_t 
\begin{pmatrix}\label{h-s}
\vr\\
u  
 \end{pmatrix}
+ 
\nabla_x \cdot
\begin{pmatrix}
P \\
\frac{P\otimes P}{\vr^2} + \log{\vr} \CI_3 
 \end{pmatrix}
=
\begin{pmatrix}
0\\
\frac{1}{\vr}  F(\vr, u, \nabla_x \Phi(x)
 \end{pmatrix}
\end{eqnarray}
with $P= \vr u$ and $F$ given in \eqref{potential}.
 
It turns out that \eqref{h-s} has the following structure of symmetric hyperbolic systems:
For all $w \in {\mathcal G},$ there is a positive definite matrix $A_0(w)$ that is smooth in $w$ and satisfies: 
$$c_0^{-1} {\CI}_4 \le A_0(w) \le c_0  {\CI}_4,$$     	
with a constant $c_0$  uniform for $w \in {\mathcal G}_1 \subset \bar{\mathcal G}_1 \subset {\mathcal G}$ 
such that 
$$A_i(w) = A_0(w)  \nabla f_i(w)$$ 
is symmetric.

Here, $\nabla f_i(w), i = 1,2,3$ are the $4\times 4$ Jacobian matrices and ${\CI}_4$ is the $4\times 4$ identity matrix. 

 The matrix
 $$A_0(w) = 
 \begin{pmatrix}
 {\vr}^{-1}      & 0\\
 0                    &\vr \mathbb{I}_3
\end{pmatrix}
$$
is called the symmetrizing matrix of system \eqref{h-s}. 
Multiplying \eqref{h-s} by $A_0$ we obtain 
\begin{equation}\label{in-h-s}
A_0(w) \partial_t w + A(w) \Grad_x w = G(w,\nabla_x\Phi)
\end{equation}
with smooth initial data
\begin{equation}
w_0= w(x,0). \label{ID-in-h-s}
\end{equation}
Here, $A(w)= (A_1(w), \dots, A_3(w)) $ denotes a matrix whose columns \\ 
$A_i(w) = A_0(w) \Grad_x f_i(w) $ are symmetric, and 
$$f(w) = 
\begin{pmatrix}
P\\
\frac{P \otimes P}{ \vr^2} + \log{\vr} \mathbb{I}_3
\end{pmatrix}
$$
with $P=\vr u$ and $\Grad_x f_i(w),$ are the $4 \times 4$ Jacobian matrices  and $\mathbb{I}_3$ denotes the $3\times 3$ identity matrix.
The $4\times 1$ vector $G$ in \eqref{in-h-s} is given 
\begin{eqnarray}\label{source}
G(w, \nabla_x \Phi)= A_0(w) 
\begin{pmatrix}
0\\
\frac{1}{\vr} F(\vr, u, \nabla_x \Phi)
\end{pmatrix}
=
\begin{pmatrix}
0\\
F(\vr, u, \nabla_x \Phi)
\end{pmatrix}
\end{eqnarray}

which implies that 
$$
G(w, \nabla_x \Phi))
= 
\begin{pmatrix}
0\\
\int_{\R^d} K(x,y)\vr(x)\vr(y)[u(y)-u(x)]~dy - \vr \Grad_x \Phi
\end{pmatrix}.
$$
We are now ready to prove the local existence of smooth solutions.

\begin{theorem}\label{local-existence}
Assume $w_0= (\vr_0, u_0) \in H^s \cap L^{\infty}(\R^3)$ with $s>5/2$ and $\vr_0(x)>0$ and that $\na_x \Phi \in H^s$.
Then there is a finite time $T \in (0,\infty),$ depending on the $H^s$ and $L^{\infty}$ norms of the initial data, such that the Cauchy problem  \eqref{eq:cont}-\eqref{eq:moment} and \eqref{ID} has a unique bounded smooth solution $w=(\vr, u) \in C^1(\R^3 \times [0, T]),$ with $\vr >0$ for all $(x,t) \in \R^3 \times [0,T],$ and $(\vr, u) \in C([0,t]; H^s) \cap C^1([0,T]; H^{s-1})$.  
\end{theorem}

Theorem \ref{local-existence} is a consequence of the following theorem on the local existence of smooth solutions,  with the specific state space ${\mathcal G}= \{(\vr,u)^\top: \vr  >0\} \subset \R^4$
for the inhomogeneous system \eqref{in-h-s}.

\begin{theorem}\label{local-smooth-in-h-s} 
Assume that $w_0: \R^d \to {\mathcal G}$ is in $H^s \cap L^{\infty}$ with $s > \frac{d}{2} + 1.$ Then, for the Cauchy problem \eqref{in-h-s}-\eqref{ID-in-h-s}, there exists a finite time $T=T(\|w_0\|_s, \|w_0\|_{L^{\infty}})\in (0,\infty)$ such that there is a classical solution $w \in C^1(\R^3 \times [0,T])$ with $w(x,t) \in {\mathcal G}$ for $(x,t) \in {\R}^d  \times [0,T]$ and $w \in C([0,T]; H^s) \cap C^1([0,T]; H^{s-1}).$
\end{theorem}
\begin{proof}
The proof of this theorem proceeds via a classical iteration scheme method. An outline of the proof of Theorem \ref{local-smooth-in-h-s} (and therefore Theorem \ref{local-existence})  is given below.
Consider the standard mollifier $$\eta(x) \in C^{\infty}_0({\R}^3),\quad \mbox{supp }\eta(x) \subseteq \{x; |x|\le1\},\quad \eta(x) \ge 0,\quad \int_{\R^3} \eta(x) dx =1,$$ and set 
$$\eta_{\epsilon} = \epsilon^{-d} \eta(x/\epsilon).$$ 
Define the initial data $w_0^k \in C^{\infty}(\R^3)$ by
$$
w_0^k(x) = \eta_{{\epsilon}_k}\star w_0(x) = \int_{\R^3} \eta_{{\epsilon}_k}(x-y) w_0(y) dy,
$$
where $\epsilon_k = 2^{-k} \epsilon_0$ with $\epsilon_0>0$ constant. 
We construct the solution of \eqref{in-h-s}-\eqref{ID-in-h-s} using the following iteration scheme.
Set  $w^0(x,t) =w_0^0(x)$ and define $w^{k+1}(x,t),$ for $k=0,2,\dots,$ inductively  as the solutions of linear equations:
\begin{equation}\label{it-sc}
\begin{cases}
A_0(w^k) \partial_t w^{k+1} + A(w^k) \Grad_x w^{k+1} = G(w^k, \nabla_x \Phi),\\
 w^{k+1}|_{t=0} = w_0^{k+1}(x).
\end{cases}
\end{equation} 
By well-known properties of mollifiers it is clear that:
$$\| w_0^k - w_0\|_s \to 0, \, \mbox{as} \, k \to \infty, \, \mbox{and}\, 
\| w_0^k - w_0\|_0 \le C_0 \epsilon_k \|w_0\|_1,$$
 for some constant $C_0.$
Moreover, $w^{k+1} \in C^{\infty}(\R^d \times [0, T_k])$ is well-defined on the time interval $[0,T_k]$,
where $T_k >0$ denotes the largest time for which the estimate $\n w^k - w_0^0\n_{s,T_k} \le C_1$ holds. 
We can then assert the existence of $T_* >0$ such that $T_k \ge T_*$ $(T_0=\infty)$ for $k=0,1,2,\dots, $ from the following estimates:
\begin{equation}\label{est}
\n w^{k+1} - w_0^0\n_{s, T_*} \le C_1, \,\, \n w_t^{k+1}\n_{s-1, T_*} \le C_2,
\end{equation}
for all $k=0,1,2,\dots,$ for some constant $C_2>0.$
From \eqref{it-sc}, we get
\begin{equation} \label{lin-bl}
A_0(w^k) \partial_t(w^{k+1} - w^k) + A(w^k) \nabla(w^{k+1} - w^k) = E_k + G_k,
 \end{equation}
where
\begin{equation}
\begin{cases}
&E_k = -(A_0(w^k) - A_0(w^{k-1})\partial_t w^k - (A(w^k) - A(w^{k-1})) \nabla w^k,\\
&G_k= G(w^k,\nabla_x \Phi).
\end{cases}
\end{equation}
From the standard energy estimate for the linearized problem \eqref{lin-bl} we get
$$\n w^{k+1} - w^k\n_{0,T} \le C e^{CT} (\|w_0^{k+1} - w_0^k\|_0 + T \n E_k \n_{0,T} +T \n G_k \n_{0,T}).$$ 
Taking into consideration the property of mollification, relation \eqref{est}, we have that
$$\|u_0^{k+1} - u_0^k\|_0 \le C 2^{-k}, \,\, \n E_k \n_{0,T} \le C \n w^k - w^{k-1}\n_{0,T}.$$ 
Note that,
\[
\begin{split}
	\n G_k \n_{0,T} &= \sup_{0 \le t \le T} \|G_k(\cdot,t)\|_s \\
	&= \sup_{0\le t\le T} \left\| \int_{\R^d} K(x,y)\vr^k(x)\vr^k(y)[u^k(y)-u^k(x)]~dy - \vr^k \Grad_x \Phi(x)\right\|_s\\
	&\le M \int_{\R^d} |\nabla \sqrt{u^k}|^2 dy + C =  M \int_{\R^d}  \frac{|\nabla u^k(y)|^2}{u^k(y)}dy + C.
\end{split}
\]
Here we use the well known result that states: For any $v \in H^1(\R^3),$
$$ \int_{\R^3} \frac{|v(x)|^2}{1+ |x|^2} dx \le M \int_{\R^3}|\nabla v|^2 dx.$$

For small $T$ such that $C^2 T \exp{CT} < 1$ one obtains
$$\sum_{k=1}^{\infty} \n w^{k+1} - w^k \n_{0,T} < \infty, $$
which implies that there exists $w \in C([0,T]; L^2(\R^3))$ such that 
\begin{equation}\label{conv}
\lim_{k \to \infty} \n w^k - w\n_{0,T} =0.
\end{equation}
From \eqref{est}, we have $\n w^k\n_{s,t} + \n w^k_t\n_{s-1, T} \le C,$ and $w^k(x,t)$ belongs to a bounded set of $G$ for $(x,t) \in \R^d \times [0,T].$ Then by interpolation we have that for any $r$ with $0 \le r < s,$
\begin{equation}\label{last}
\n w^k - w^l \n_{r,T} \le C_s \n w^k - w^l \n^{1-r/s}_{0,T} \n w^k - w^l \n_{s,T}^{r/s} \le \n w^k - w^l\n_{0,T}^{1-r/s}.
\end{equation}
From \eqref{conv} and \eqref{last},
$$
\lim_{k \to \infty} \n w^k - w\n_{r.,T} =0
$$
 for any $0<r<s.$ Therefore, choosing $r > \frac{3}{2} +1$, 
Sobolev's lemma implies
\begin{equation}\label{cnv}
w^k \to w \,\, \mbox{in} \,\, C([0,t]; C^1(\R^3)). 
\end{equation} 
From \eqref{lin-bl} and \eqref{cnv} one can conclude that $w^k \to w$ in $C([0,T]; C(\R^3),$ $w \in C^1(\R^3\times [0,T]),$ and $w(x,t)$ is a smooth solution of   \eqref{in-h-s}-\eqref{ID-in-h-s}.
To prove $u \in C([0,T];H^s) \cap C^1([0,T];H^{s-1})$, it is sufficient to prove $u \in C([0,T]; H^s)$ since it follows from the equations in \eqref{in-h-s} that $u \in C^1([0,T]; H^{s-1}).$
\end{proof}
\begin{remark}
The proof of Theorem \ref{local-existence} follows the line of argument presented by Majda \cite{Majda-1984}, which relies solely on the elementary linear existence theory for symmetric hyperbolic systems with smooth coefficients (see also Courant-Hilbert \cite{C-H-1953}).  
\end{remark}
\iffalse
The energy/entropy for the limiting system is given by
\begin{equation}
	E(u) =  \frac{\vr |u|^2}{2} + \vr \log \vr + \vr \Phi
\end{equation}
\fi

\begin{lemma}\label{lem:}
	Let  $(\vr, u)$ be a sufficiently smooth solutions  to 
	\eqref{eq:cont}-\eqref{eq:moment}. The energy $E(u)$ satisfy
	the following entropy equality
	\begin{equation} \label{en-e}
		\partial_t \int_{\R^d} E(u)~dx  + \frac{1}{2}\int_{\R^d} K(x,y)\vr(x)\vr(y)\left|u(x)-u(y)\right|^2~dydx = 0.
	\end{equation}
\end{lemma}
\begin{proof}
The result is obtained by the following standard process: first we multiply the continuity equation \eqref{eq:cont0}  by $(\vr \log{\vr})'$
and the momentum equation \eqref{eq:mom} by the velocity field $u$, and we add  the resulting relations. Next we integrate over the domain taking into account that $(\vr,u)$ is sufficiently smooth  and using that 
$$\partial_t \int_{\R^d} \vr \, \Phi dx = \int_{\R^d} \partial_t \vr \, \Phi dx = -\int_{\R^d} \Div(\vr u)\Phi \, dx = 
\int_{\R^d} \int_{\R^d} \vr u \nabla_x \Phi dx.$$
and 
$$
\int_{\R^d} u(x) \int_{\R^d}  K(x,y) \vr(x) \vr(y) [u(y) - u(x)] dy dx  
$$
$$
\quad \quad\quad \quad \quad \quad \quad =  \frac{1}{2} \int_{\R^d}\int_{\R^d}K(x,y)\vr(x)\vr(y)\left|u(x) - u(y)\right|^2~dy\, dx
$$
we arrive at \eqref{en-e}.
The result for  a solution $(\vr, u)$ with the regularity established in Theorem \ref{local-existence} is established using a density argument.
\end{proof}
\iffalse
\begin{proposition}\label{pro:}
There is an existence time $T^*>0$ and functions $(\vr, u)$
such that $(\vr, u)$ is a strong solution of \eqref{eq:cont} - \eqref{eq:moment}.
Moreover, $(\vr, u)$ satisfies the inequality 
\begin{equation*}
	\begin{split}
		&\sup_{t \in (0,T^*)}\int_{\R^d} \vr \log_+ \vr + \vr \frac{|u|^2}{2} + \vr \Phi~dx
		+\frac{1}{2}\int_{\R^d} K(x,y)\vr(x)\vr(y)\left|u(x)-u(y)\right|^2~dydx \\
		&\qquad \leq \int_{\R^d} \vr_0 \log \vr_0 + \vr_0 \frac{|u_0|^2}{2} + \vr_0 \Phi~dx.
	\end{split}
\end{equation*}
\end{proposition}
\begin{proof}
The proof of Theorem \ref{local-existence} can be further reduced to verifying that $u(x,t)$ is strongly right continuous at $t=0,$ since the same argument works for the strong right-continuity at any other  $t \in (0,T]$  and the strong right-continuity on $[0,T)$  implies the strong left-continuity on $(0,T]$ because the equations in \ref{in-h-s} are reversible in time. Taking this into consideration  estimate  \eqref{en-e} implies the result.
\end{proof}
\fi

%\noindent
%{\bf Proof of Theorem \ref{thm:Euler}}
\subsection{Proof of Theorem \ref{thm:Euler}}
The proof of Theorem \ref{local-existence} can be further reduced to verifying that $u(x,t)$ is strongly right continuous at $t=0,$ since the same argument works for the strong right-continuity at any other  $t \in (0,T]$  and the strong right-continuity on $[0,T)$  implies the strong left-continuity on $(0,T]$ because the equations in \ref{in-h-s} are reversible in time. Taking this into consideration  estimate  \eqref{en-e} implies the result. $\square$

%\addcontentsline{toc}{section}{References}

\end{document}